\documentclass[12pt]{amsart}

\newcommand{\ra}{\rightarrow}
\newcommand{\R}{\mathbb{R}}

\newcommand{\Z}{\mathbb{Z}}

\newcommand{\ee}{\end{eqnarray}}
\newcommand{\be}{\begin{eqnarray}}

\newtheorem{thm}{Theorem}[section]
\newtheorem{lem}[thm]{Lemma}
\newtheorem{definition}[thm]{Definition}

\usepackage{latexsym}
\usepackage{graphicx}
\usepackage{graphics}
\usepackage{epsfig,longtable}
\usepackage[tight,hang]{subfigure}
\usepackage{amsmath}
\usepackage{amssymb}
\numberwithin{equation}{section}
\begin{document}

\title[Homotopical Minimal Measures on Surfaces of Higher Genus]{Homotopical Minimal Measures for Geodesic flows on Surfaces of Higher Genus}
\author{Fang Wang}
\address{School of Mathematical Sciences, Capital Normal University, Beijing, 100048, China.}
\email{fangwang@cnu.edu.cn}

\author{Zhihong Xia}
\address{Department of Mathematics, Northwestern University, Evanston, IL 60208, USA.}
\email{xia@math.northwestern.edu}

\begin{abstract} We study the homotopical minimal measures for positive definite autonomous Lagrangian systems. Homotopical minimal measures are action-minimizers in their homotopy classes, while the classical minimal measures (Mather measures) are action-minimizers in homology classes. Homotopical minimal measures are much more general, they are not necessarily homological action-minimizers. However, some of them can be obtained from the classical ones by lifting them to finite-fold covering spaces. We apply this idea of finite covering to the geodesic flows on surfaces of higher genus. Let $(M,G)$ be a compact closed surface with genus $g>1$, where $G$ is a complete Riemannian metric on $M$. Consider the positive definite autonomous Lagrangian $L(x,v)=G_x(v,v)$, whose Lagrangian system $\phi_t: TM\ra TM$ is exactly the complete geodesic flow on $TM$.  We show that for each homotopical minimal ergodic measure $\mu$ that is supported on a nontrivial simple closed periodic trajectory, there is a finite-fold covering space $M'$ such that each ergodic preimage of $\mu$ on $TM'$ is a minimal measure in the classic Mather theory for the Lagrangian system on $TM'$.

\end{abstract}

\maketitle
\section{INTRODUCTION}\label{chapter1}

Suppose $(M,G)$ is a closed
compact connected orientable manifold with a complete Riemannian metric $G$. Let $L: TM \rightarrow \mathbb{R}$ be a smooth function, called the
Lagrangian, which satisfies the following three conditions:

\begin{enumerate}
\item {\em Positive definiteness}\/: For each $x \in M$, and $v \in T_xM$, the restriction of $L$ to $T_xM$ is strictly
convex, in the sense that its Hessian second derivative $(\partial^2 L(x, v)/ \partial v_i \partial v_j)$ is everywhere positive definite.

\item {\em Super-linear Growth}\/: $$\frac{L(x, v)}{\| v \|} \rightarrow
+\infty, \; \mbox{ as } \; \| v \| \rightarrow \infty$$ for all $x\in M$, where $\| \cdot \|$ is
the Riemannian metric on $M$.

\item {\em Completeness of the Euler-Lagrange flow}\/: All the
solutions of the Euler-Lagrange equation can be extended to all $t
\in \mathbb{R}$.
\end{enumerate}
\vspace{2ex}

The (autonomous) Euler-Lagrange flow is a smooth flow $\varphi_{t} : TM\hookleftarrow$ generated by the Euler-Lagrange equation, which is, in local coordinates:
$$\frac{d}{dt} \frac{ \partial L}{ \partial \dot{x}}(x, \dot{x}) =
\frac{\partial L}{ \partial x}(x, \dot{x}).$$
The Euler-Lagrange equation is a first-order ordinary differential equation on $TM$, which defines a vector field on $TM$ and then generates the flow $\varphi_{t} : TM\hookleftarrow$.
By the assumption that the Lagrangian $L$ satisfies conditions (1-3), the flow $\varphi_{t}$ is usually called a positive definite Euler-Lagrange flow or a positive definite Lagrangian system.\\

The geodesic flow is a primary example of positive definite Lagrangian system. Let $$L(x,v)=G_x(v,v).$$ One can check that $L$ satisfies the three conditions in the above, and the corresponding Euler-Lagrange equation is equivalent to the geodesic equation.  Therefore a geodesic flow is a positive definite Euler-Lagrange flow.  The geodesic flows, especially the geodesic flows on manifolds with negative or non-positive curvatures, are very typical conservative dynamical systems, and often used as touchstones to validate new theories and methods in general conservative systems. In this work, we will follow this traditional technical route, and employ our ideas in the study of the minimal measures in the homotopical version to the geodesic flows on compact surfaces with higher genus.  Some conjectures proposed in our previous work on the Mather theory will be show to valid for geodesic flows on surfaces.    \\

Mather theory studies the action-minimizing trajectories and minimal measures for positive definite Lagrangian systems on compact manifolds. In the theory of analytic mechanics, a basic principle is that a true trajectory of
a mechanical system must be an extremal of action, i.e., a trajectory
$\gamma(t)$ must satisfy the following fixed endpoints variational equation:
$$\delta \int_a^b L(\gamma(t), \gamma'(t))~dt =0.$$
It is a easy to check that these true trajectories are all solutions of the
Euler-Lagrange equation. Mather theory was established exactly based on this variational principle. \\

The basic idea in Mather theory is to classify the invariant measures with their associated first homology class or cohomology class (cf.~\cite{Mat2}). Accordingly, the action-minimizing trajectories defined in the Mather theory are the trajectories obtained by consider the fixed endpoint variational equation under a homologous constrain. The advantage of using the first homology to classify the invariant measures is that the first homology group $H_1(M,\mathbb{R})$ is a vector space, on which we can do calculations. However, when the manifold $M$ has complicated topological structure, for example when the fundamental group $\pi_1(M,x)$ is not commutative, the Lagrangian systems defined on $TM$ may have richer dynamics which are not completely discovered by Mather theory. This is the reason we expect to consider a more general framework of action-minimizers and minimal measures.

For the purpose of exploring the dynamical complexity of Lagrangian systems on manifolds with complicated topology, the fundamental group seems to be a suitable candidate for classifying invariant measures. There have been many evidences suggesting that the complicity in the topology of $M$, especially the complicated structure of the fundamental group $\pi_1(M,x)$, leads to the complicity in dynamics for Lagrangian systems on $M$. For example, in \cite{Dina}, Dinaburg showed that, for geodesic flows on the complete Riemannian manifold, if the fundamental group of the manifold has exponential growth, then the topological entropy of the geodesic flow is positive. A lower bound of the topological entropy of the geodesic flow is given in terms of the algebraic growth rate of the fundamental group. This result was also proven to be valid for general autonomous positive definite Lagrangian systems (cf.~\cite{LWW}). This implies that a substantial part of the complicity in dynamics of the autonomous Lagrangian systems is contributed by the topological complicity of $M$ relating to the fundamental group.  Moreover, we also notice that for many compact manifolds, such as  manifolds which admit Riemannian metrics with negative sectional curvatures, the topology of the manifolds can be uniquely determined by the fundamental groups (cf.~\cite{FJ}). It is very promising that by studying minimal measures in the homotopical version, we can obtain more comprehensive information of the Lagrangian systems. \\

The disadvantage of using $\pi_1(M,x)$ instead of $H_1(M,\mathbb{R})$ to classify invariant measures, is that there is no vector space structure on $\pi_1(M,x)$ when it is non-commutative. This creates a significant obstacle when we try to use the existing results and techniques in the classical Mather theory. To overcome this difficulty, in \cite{WX}, we propose a method of considering the lifting the Lagrangian systems to some suitable finite covering spaces, and in this way, "transform" the minimality in the homotopical version to the minimality in the homological version on the covering spaces. It has been proven in \cite{WX} that it is possible to transform some minimal ergodic measures in the homotopical version to Mather's minimal measures through finite lifting. We conjecture that (under suitable conditions) a minimal ergodic measure in the homotopical version can either be lifted to a Mather measure on some finite-fold covering space, or be a limit of a sequence  of invariant measures that are projections of Mather measures on finite-fold covering spaces.\\

In this article, we will show that our conjecture is at least partially valid for geodesic flows on surfaces of higher genus. We will prove every homotopical minimal ergodic measure distributed on a closed trajectory can be lifted to a Mather's minimal measure on a finite covering space. We have the following Theorem.

\begin{thm}[Main Theorem]\label{main theorem}
Suppose $(M,G)$ is a compact closed surface of genus greater than $1$, which possesses a complete Riemannian metric $G$. Let $\varphi_{t}$ be the positive definite Lagrangian system on $TM$ generated by the  Lagrangian $L(x,v)=G_x(v,v),$ i.e. $\varphi_{t}$ is the complete geodesic flow defined on $TM$. For each minimal ergodic measure $\mu$ in the homotopical version which is distributed on a closed trajectory, there is a finite-fold covering space $M'$ with covering map $p: M' \ra M$ such that every ergodic preimage of $\mu$ on $T M'$ is a minimal measure in Mather theory for the lifted Lagrangian systems generated by $L'=L\circ dp$.
\end{thm}

We remark that some of the techniques used in the proof are from our previous works \cite{W} and \cite{WX}. We will recall all the key results we need, in Chapter \ref{chapter4} and \ref{chapter5}. Moreover, a discussion of non-periodic minimal measures in the homotopical version  for geodesic flows on surfaces with negative or non-positive curvature will be presented in the last chapter of this article. \\

\section{MATHER THEORY}\label{chapter2}

In this chapter, we give a short introduction to the classic Mather's theory for positive definite autonomous Lagrangian systems on compact manifolds.
Suppose $M$ is a smooth compact Riemannian manifold without boundary, and $L: TM\ra \R$ is a $C^2$ Lagrangian function which satisfies the three conditions in Chapter \ref{chapter1}. Let $\varphi_{t} : TM\hookleftarrow$ be the Euler-Lagrange flow generated by $L$. \\

For an absolutely continuous curve $\gamma: [a,b]\ra M$, $b>a$, we define its action $A_L(\gamma)$ in the following way:
\begin{equation*}
A_L(\gamma)=\int_a^b L(\gamma(t),\gamma'(t))~dt.
\end{equation*}

\begin{definition}[Action-mimimizer]
An absolutely continuous curve $\gamma: [a,b]\ra M$ is said to be an \emph{action-minimizer} if for any absolutely continuous curve $l: [a,b]\ra M$, with $l(a)=\gamma(a)$ and $l(b)=\gamma(b)$
and homologous to $\gamma$ relative to endpoints, we have $$A_L(\gamma)\leq A_L(l).$$ An absolutely continuous curve $\gamma: \R \ra M$ is said to be an \emph{action-minimizer} if any finite segment of $\gamma$ is an action-minimizer.
\end{definition}
Obviously, if  $\gamma: [a,b]\ra M$ is a an action-minimizer, it should satisfy the variational principle: $$\delta \int_a^b L(\gamma(t), \gamma'(t))~dt =0,$$ and therefore it is a solution of the Euler-Lagrange equation (so $\gamma$ is automatically $C^2$).  Therefore, if $\gamma: \R\ra M$ is an action-minimizer, then it is smooth and $$(\gamma,\gamma'):\R \ra TM$$ is a true trajectory of the Lagrangian system. We call $(\gamma,\gamma')$ an \emph{action-minimizing trajectory}. We remark that, the existence of action-minimizers subjected to the fixed endpoints for positive definite Lagrangian is guaranteed by the Tonelli Theorem (cf.~\cite{Mat2}). For this reason, the action-minimizers are often called Tonelli minimizers in some references.\\

The main idea of Mather theory in studying invariant measures is to classify  measures with their associated homology or cohomology classes. Suppose $\mu$ be a $\varphi_{t}$ invariant Borel probability measure on $TM$. We define its {\em average action} $A_L(\mu)$ as
$$A_L(\mu) = \int_{TM} L\ d\mu.$$ This integral is bounded below (since $L$ is bounded below) and can be
$+\infty$. If $A_L(\mu)$ is finite, we can define its {\em rotation vector}\/: $\rho(\mu) \in H_1(M, \mathbb{R})$.
\begin{definition}[Rotation vector]
For each invariant probability measure $\mu$ with finite action, the rotation vector $\rho(\mu)\in H_1(M,\R)$ is the unique real first homology class satisfying that, for every $c\in H^1 (M,\R)$ and closed
$1$-form $\lambda_c$ whose cohomology class is $c$, $$<c, \rho(\mu)> \; = \int_{TM} \lambda_c\ d\mu.$$
Here the bracket $<\cdot , \cdot>$ is the canonical pairing of $H^1(M, \mathbb{R})$
and $H_1(M, \mathbb{R})$.
\end{definition}

We remark that the integral $\int_{TM} \lambda_c\ d\mu$ does not depend on the choice of the 1-form $\lambda_c$, provided $\mu$ is an invariant probability measure, since the integral of exact 1-form with respect to invariant probability measure is always zero (cf.~\cite{Mat2}).\\

Based on the Birkhoff theorem, we can see that the rotation vector $\rho(\mu)$ of an invariant measure $\mu$ describes the average of the asymptotic speed and direction, indicated by the first homology, over generic trajectories.  In \cite{Mat2}, Mather showed that for every $h\in H_1(M,\R)$, there are invariant measures with finite action whose rotation vectors are $h$. Moreover, there is at least one invariant measure $\mu$ with $\rho(\mu)=h$, whose action is minimal among all invariant measures with rotation $h$. We denote this minimal action $\beta(h)$ for each $h\in H_1(M,\R)$.  In this way we defines a real value function $\beta=\beta_L: H_1(M,\R) \ra \R$. This is the so called \emph{Beta function}. It can be shown that $\beta(h)$ is a continuous convex function with sup-linear growth on $H_1(M,\R)$. For each $h$, we call $\beta(h)$ the {\em minimal average action}\/ of the rotation vector
$h$, and call the measure $\mu$ whose rotation vector is $h$ and action is qual to $\beta(h)$ a minimal measure. This is the following definition:

\begin{definition}[Minimal Measure]
An invariant probability measure which achieves the minimal average action of its rotation vector is called a minimal measure.
\end{definition}
For each $h\in H_1(M, \mathbb{R})$, we use $\mathfrak{M}_h$ to denote the set of all minimal measures with rotation vector $h$. We emphasize that $\mathfrak{M}_h\neq \emptyset$ for every $h\in H_1(M, \mathbb{R})$. Denote $$\mathcal{M}_L=\bigcup_{h\in H_1(M, \mathbb{R})}\mathfrak{M}_h,$$ which is the set of all minimal measure of the Lagrangian systems generated by the Lagrangian function $L$. \\

Mather exhibited many important properties of minimal measures in \cite{Mat2}, Here we list some of them, which will be used in the following.
\begin{enumerate}
\item (Lipschitz Graph Property) The support of a minimal measure can be expressed as the graph of a Lipschitz map from a subset of $M$ to $TM$.
\item All trajectories in the support of a minimal measure are action-minimizing trajectories.
\item For each $h\in H_1(M,\R)$, there is a minimal measure with rotation vector $h$ which is the weak-* limit of a sequence of probability measures evenly distributed on segments of action-minimizing trajectories.
\end{enumerate}

\vspace{.5cm}

\section{MINIMAL MEASURES IN THE HOMOTOPICAL VERSION}\label{chapter3}
From now on, we focus on the case that the compact manifold $M$ has a non-commutative fundamental group. In this case, $\pi_1(M,x)$ is not isomorphic to the first homology group $H_1(M,\Z)$. Therefore, the Abelian covering space $\tilde{M}$ is different from the universal covering space $\bar{M}$. Recall that, in Mather theory, an action-minimizers is defined to have minimal action among all absolutely continuous curves homologous to it subjected to the end points. So the set of action-minimizers is obviously a small subset of all trajectories. To get more information of the dynamics, we need to extend Mather's definition of action-minimizer and consider a more general class of "action-minimizers" which admits much more trajectories. It is a very natural idea is to consider the action-minimizer in the homotopical version. In \cite{WX} we defined the action-minimizer in the homotopical version in the following way:

\begin{definition}[Action-minimizer in the homotopical version]
An absolutely continuous curve $\gamma: [a,b]\ra M$ is said to be an action-minimizer in the homotopical version if for any absolutely continuous curve $l: [a,b]\ra M$, with $l(a)=\gamma(a)$ and $l(b)=\gamma(b)$
and homotopic to $\gamma$ relative to endpoints, we have $$A_L(\gamma)\leq A_L(l).$$ An absolutely continuous curve $\gamma: \R \ra M$ is said to be an action-minimizer in the homotopical version if any finite segment of $\gamma$ is an action-minimizer in the homotopical version.
\end{definition}

It is also easy to check that an action-minimizer in the homotopical version satisfies the  fixed endpoints variational principle and therefore is a solution of the Euler-Lagrange equation. So if $\gamma: \R \ra M$ is said to be an action-minimizer in the homotopical version, then $(\gamma, \gamma')$ is a true trajectory of the Lagrangian system. Similarly, we call this type of  trajectories the action-minimizing trajectories in the homotopical version.
It is easy to see that when $\pi_1(M,x)$ is non-commutative, the set of action-minimizing trajectories in the homotopical version is strictly larger than the set of Mather's action-minimizing trajectories. The set of action-minimizing trajectories in the homotopical version can be very large. For example, for the geodesic flows on a compact manifolds with negative curvature, all trajectories are action-minimizing in the homotopical version, in contrast to the fact that only a few trajectories are action-minimizing in Mather's definition.\\

Speaking of the invariant measures, we also defined in \cite{WX} the \emph{minimal measures in the homotopical version}. This definite is given by describing the action-minimizing property of the trajectories in the supports of measures.
\begin{definition}
An invariant probability Borel measure $\mu$ is called a minimal measure in the homotopical version or a homotopical minimal measure, if all trajectories in its support are action-minimizing trajectories in the homotopical version.
\end{definition}

 We denote the set of all minimal measures in the homotopical version as $\mathcal{M}'_{L}$. We can see that the set $\mathcal{M}'_{L}$ is closed under the weak-* topology.
\begin{lem}
$\mathcal{M}'_{L}$ is a convex closed set in the space of invariant probability measures on $TM$.
\end{lem}
\begin{proof}
The convexity is straightforward from the definition since $\mbox{supp}(\frac{1}{2}(\mu_1+\mu_2))=\mbox{supp}(\mu_1)+\mbox{supp}(\mu_2)$ for any minimal measures $\mu_1,~\mu_2$ in the homotopical version. \\

Suppose $\mu_1,~\mu_2,~\cdots$ are a sequence of minimal measures in the homotopical version, and $\mu_n\rightarrow \mu$ under the weak-* topology as $n\rightarrow\infty$. It is sufficient to prove that $\mu$ is also a minimal measure in the homotopical version. \\

We prove this by contradiction. Assume $\mu$ is not a minimal measure in the homotopical version. Then there is a trajectory $(\gamma,\gamma')$ in the support of $\mu$ which is not action-minimizing in the homotopical version, i.e. $\gamma$ is not an action-minimizer in the homotopical version. This implies that $\exists~ a<b\in\R$ such that $\gamma|_{[a,b]}$ is not an action-minimizer. Denote $\gamma_0=\gamma|_{[a,b]}: [a,b]\ra M$. By Tonelli Theorem (cf.~\cite{Mat2}), there is a smooth curve $l_0:[a,b]\ra M$ homotopic to $\gamma_0$ relative to endpoints, which is an action-minimizer in the homotopic version. Let $\delta=A_L(\gamma_0)-A_L(l_0)$. Obviously $\delta>0$, since $\gamma_0$ is not an action-minimizer in the homotopical version.\\

Let $x=\gamma(a),~y=\gamma(b)\in M$ and $\textbf{p}=(x,\gamma'(a)), \textbf{q}=(y,\gamma'(b))\in TM$. Since both $\gamma_0$ and $l_0$ are solutions of the Euler-Lagrange equation and $\gamma_0(a)=\l_0(a),~\gamma_0(b)=\l_0(b)$, we know $$\gamma'_0(a)\neq\l'_0(a),~\gamma'_0(b)\neq\l'_0(b).$$  By the continuity of the Euler-Lagrange flow, there are small open neighborhoods $\textbf{U},~\textbf{V}$ of $\textbf{p},~\textbf{q}$ respectively, satisfying the following properties:
\begin{enumerate}
\item If a trajectory $(s,s')$ of the Euler-Lagrange flow satisfies $(s(a),s'(a))\in\textbf{U}$, then $(s(b),s'(b))\in \textbf{V}$.
\item There is an $\epsilon>0$ such that $B((l_0(a),l'_0(a)),\epsilon)\cap\textbf{U}=\emptyset,~B((l_0(b),l'_0(b)),\epsilon)\cap\textbf{V}=\emptyset$.
\end{enumerate}
Since the minimal action is continuous with respect to the end points, by shrinking $\textbf{U},~\textbf{V}$ if needed, we can guarantee that none of the trajectories passing through $\textbf{U}$ is action-minimizing in the homotopical version.\\

Since $\textbf{p}\in\mbox{supp}(\mu)$, we know $\mu(\textbf{U})>0$, no matter how small $\textbf{U}$ is. By $\mu_n\rightarrow \mu$, we know that there is a $N>0$  such that $\mu_n(\textbf{U})>\frac{\mu(\textbf{U})}{2}>0$ for all $n\geq N$. Fix a $n\geq N$, obviously $\mbox{supp}(\mu_n)\cap \textbf{U}\neq\emptyset$. So there is a trajectory $(s,s')$ in the support of $\mu_n$ which goes through $\textbf{U}$. By the discussion above, $(s,s')$ is not action-minimizing in the homotopical version. This contradicts to the fact the $\mu_n$ is a minimal measure in the homotopical version. We are done with this proof.
\end{proof}

Sometimes, for comparison, we also call the minimal measures in Mather theory the \emph{minimal measures in the homological version}. Obviously a minimal measures in the homological version is also a minimal measure in the homotopical version, i.e. $$\mathcal{M}_L\subset\mathcal{M}'_L.$$ This inclusion is proper in general.  For example, if $(\gamma, \gamma')$ is a closed action-minimizing trajectory in the homotopical version but not action-minimizing in the classic Mather theory, then the ergodic measure evenly distributed on $(\gamma, \gamma')$ is only a minimal measure in the homotopical version but not a Mather's minimal measure. It is easy to see that when $\pi_1(M,x)$ is not isomorphic $H_1(M,\Z)$, we do have a lot of ergodic minimal measures in the homotopical version which is not minimal in Mather theory.\\

Basic properties of action-minimizers and minimal measures in the homotopical version can be found in \cite{WX}. To further understand their characteristics, we presented an idea in \cite{WX}: lifting the Lagrangian system to a well-chosen finite-fold covering space, it is possible to "transform" a minimal ergodic measure in the homotopical version to a minimal measure in the homological version. Then the classical Mather theory can be employed in studying its properties. To realize this idea, we need to study the lifted Lagrangian systems on finite-fold covering spaces.
\vspace{.3cm}

\section{FINITE COVERING SPACES AND MINIMAL MEASURES}\label{chapter4}

In this section, we consider the minimal measures on finite-fold covering spaces. We assume that $M$ is a compact closed manifold whose fundamental group is non-commutative, and $\hat{M}$ is a $k$-fold covering space of $M$, ($k>1$), with the covering map $p: \hat{M} \ra M$. Obviously $\hat{M}$ is also a compact closed manifold. The covering map $p$ also induces a cover map $dp:T\hat{M} \ra TM$.\\

Suppose $L$ is a Lagrangian function satisfying the three conditions stated in Chapter \ref{chapter1}, which generates a positive definite Lagrangian system $\varphi_{t} : TM\hookleftarrow$. Consider the lifting $\hat{L}=L\circ dp : T\hat{M}\ra \R$. It is easy to check $\hat{L}$  also satisfies the three conditions of positive definiteness, super-linear growth, and the completeness of the Euler-Lagrange flow. So $\hat{L}$ also generates a positive definite Lagrangian system $\hat{\varphi}_{t} : T\hat{M}\hookleftarrow$. Notice that in fact the Lagrangian system $\hat{\varphi}_{t}$ on $T\hat{M}$ is exactly a lifting of the Lagrangian system $\varphi_t$ in the sense that, for each $x\in M$ and $v\in T_xM$, if $\hat{x}\in \hat{M}$ is a preimage of $x$ under $p$ and $\hat{v}\in T_{\hat{x}}\hat{M}$ is a preimage of $v$ under $dp$, then for all $t\in\R$, $$dp\circ\hat{\varphi}_{t}(\hat{v})=\varphi_t(v).$$

We expect to compare the set of minimal measures for the Lagrangian systems $\varphi_t$ and its lifting $\hat{\varphi}_t$ on the covering space. There is a standard projection $dp_*$ induced by the covering map $dp$, which sends probability Borel measures on $T\hat{M}$  to probability Borel measures on $TM$. Moreover it is easy to check this projection is surjective.  Now we use $\mathfrak{M}_{inv}(M)$ and  $\mathfrak{M}_{erg}(M)$ to denote the set of invariant probability measures and ergodic measures on $TM$ respectively, and use $\mathfrak{M}_{inv}(\hat{M})$ and  $\mathfrak{M}_{erg}(\hat{M})$ to denote the set of invariant probability measures and ergodic measures on $T\hat{M}$ respectively. We showed the following results in \cite{WX}:
\begin{enumerate}
\item For each probability measure $\hat{\mu}$ on $T\hat{M}$, suppose $\mu=dp_*(\hat{\mu})$, then $$A_{\hat{L}}(\hat{\mu})=A_L(\mu).$$
\item $dp_*$ sends  $\mathfrak{M}_{inv}(\hat{M})$  to  $\mathfrak{M}_{inv}(M)$ surjectively.
\item $dp_*$ sends  $\mathfrak{M}_{erg}(\hat{M})$  to  $\mathfrak{M}_{erg}(M)$ surjectively.
\item For each $\mu\in\mathfrak{M}_{erg}(M)$, its preimage set $dp_*^{-1}(\mu)$ has at most $k$ distinct ergodic measures.
\end{enumerate}

Considering the group homomorphism $p_*: H_1(\hat{M}, \R)\ra H_1(M,\R)$ induced by the covering map $p$, we know that $p_*$ is surjective. We can establish the relation between the rotation vectors of invariant measures on $T\hat{M}$ and the rotation vectors of their projections on $TM$. Suppose $\hat{\mu}\in \mathfrak{M}_{inv}(\hat{M})$  and $\mu=dp_*(\hat{\mu})$ is its projection on $TM$, then we can check that:$$p_*(\rho(\hat{\mu}))=\rho(\mu).$$ Moreover, for each $h\in H_1(M,\R)$, Let $\hat{V}_h=p_*^{-1}(h)\subset H_1(\hat{M}, \R)$ be the set of preimages of $h$ (it is an affine subspace in $H_1(\hat{M}, \R)$), then from the above equality and the definition of minimal measures, we can see that:
\begin{equation*}
\beta_L(h)=\min\{\beta_{\hat{L}}(\hat{h})~|~\hat{h}\in \hat{V}_h\}.
\end{equation*}
This means that if $\mu$ is a minimal measure of $\varphi^t$ on $TM$, then any of its preimage is a minimal measure of on $\hat{\varphi}_{t}$ on $T\hat{M}$, i.e.~ $dp_*^{-1}$ \emph{lifts minimal measures to minimal measures}.\\

Notice that, since $\hat{V}_h$ is an affine subspace in $H_1(\hat{M}, \R)$, the set $\{\beta_{\hat{L}}(\hat{h})~|~\hat{h}\in \hat{V}_h\}$ is unbounded from above. So, there are infinitely many preimages $\hat{h}$ of $h$,
whose minimal measures on $T\hat{M}$ do not project to minimal measures on $TM$. However, they project to minimal measures in the homotopical version on $TM$ (cf. \cite{WX}). Let $\mathcal{M}^{*}_{L}$ be the set of all invariant measures which are projections of Mather's minimal measures for lifted Lagrangian systems on finite-fold covering spaces. Based on the discussion above, we can draw the conclusion that $$\mathcal{M}_L\subset \mathcal{M}^{*}_L\subset\mathcal{M}'_L.$$

It has been showed, for geodesic flows on manifolds with non-commutative fundamental groups, $\mathcal{M}_L$ is a proper subset of $\mathcal{M}^*_L$ (cf.~\cite{WX}). We think that in general the set $\mathcal{M}^*_L$ should be very close to $\mathcal{M}'_L$ in the sense that all ergodic elements of $\mathcal{M}'_L$  are contained in $\mathcal{M}^*_L$ or in $\overline{\mathcal{M}^*_L}$, where $\overline{\mathcal{M}^*_L}$ denotes the closure of $\mathcal{M}^*_L$ in the space of probability measures on $TM$. In this paper we have prove this for a large part of minimal ergodic measures in homotopical version on are projections of Mather's minimal measures on finite-fold covering spaces, for geodesic flows on surfaces of higher genus.\\

\section{MINIMAL MEASURE FOR GEODESIC FLOWS}\label{chapter5}
We will give of proof of Theorem \ref{main theorem} in the next section. As a preparation, we should go over some preliminary results for geodesic flows on compact surfaces, which will be extensively used in the discussion of the next chapter. We will not give the proofs for the results in this section. For more details, please refer to \cite{W}.\\

Suppose $M$ is a smooth closed compact connected
orientable surface with genus $g>1$. Obviously for any $x\in M$, the fundamental group $\pi_1(M,x)$ is non-commutative, and has an exponential (algebraic) growth rate. Let $G$ be a complete Riemannian metric on $M$. Consider the
Lagrangian $$L(x,v)=G_{x}(v,v).$$  By the variational principle, it is easy to check that in the local
coordinates, the geodesics are precisely the maximal solutions of
the Euler-Lagrange equation for the
Lagrangian function $G$ (cf. ~\cite{Pa}). Therefore the positive definite Lagrangian system generated by $G$ is exactly the geodesic flow on $TM$.
We remark that in this paper, without specifical
indication, we suppress the notional distinct between a geodesic $l$
on $M$ and the trajectory $(l, l')$ of the geodesic flow on $TM$, and call both of them
geodesics.\\

Given an absolutely continuous curve $l: [a,b]\rightarrow M$. Let $|l|$ denote the arclength of $l$, i.e. $$|l| = \int_{a}^{b} \| l'(t)\| dt=\int_{a}^{b} \sqrt{G_{l(t)}(l'(t),l'(t))} dt.$$
The action $A_L(l)$ of $l$ is defined to be:
$$A_L(l)=\int_{a}^{b}G_{l(t)}(l'(t),l'(t)) dt = \int_{a}^{b} \| l'(t)\|^2 dt.$$
If in addition $l$ is a closed curve and $h=[l]\in H_1(M,\Z)$, in \cite{Mat2} Mather defined its rotation vector $\rho(l)$ as $$\rho(l)= \frac{h}{T} \in H_{1} (M, \R),$$
where $T=b-a>0$. Here we identify $H_{1}(M, \Z)$ with the lattice of integral
vectors in $H_{1}(M, \R)$.\\

Suppose $\gamma: \mathbb{R} \rightarrow M$ is a periodic geodesic with period $T>0$, and $l=\gamma|_{[0,T]}:[0,T]\rightarrow M$. Let $\mu_{l}$ denote the
unique invariant probability measure evenly distributed on $\{
(l(t), l'(t))\}\mid_{[0,T]}$, then it is easy to check that
$$\rho(\mu_{l})=\rho(l)=\frac{h}{T},~~~~ \  \&~~ \ A_L(\mu_{l})= \frac{1}{T} \int_{0}^{T} \| l'(t)\|^2 dt =
\frac{1}{T} A_L(l).$$

Let $\tilde{l}: [0, \frac{T}{a}] \rightarrow M$ be a re-parameterization
of $l$ with $\|\tilde{l}'(t)\| =a \| l'(t)\|$, It is straightforward that $$\rho(\tilde{l})=\frac{a}{T}h= a
\rho(l),~~~\ \&~~~\ A_L(\tilde{l})=\int_{0}^{\frac{T}{a}}(a\| l'(t)\|)^{2} dt = a A_L(l).$$
Then the rotation vector and the action of $\mu_{\tilde{l}}$ are
$$\rho(\mu_{\tilde{l}})=\rho(\tilde{l})= a
\rho(l),~~~\ \&~~~\ A_L(\mu_{\tilde{l}})=\frac{a}{T}A_L(\tilde{l})=\frac{a^2}{T}A_L(l)=a^{2}A_L(\mu_{l}).$$

Similar results are also valid for general invariant probability measures. Suppose $\mu$ is an invariant probability measure of the geodesic flow and $a>0$ is a constant. Let $\mu'$ be the
probability measure satisfying the following property: for every measurable set $E'$
and $E=\{(x,v)\in TM \mid (x, av) \in
E'\}$, we have $\mu'(E')=\mu(E).$ $\mu'$ is called {\em a shifting of $\mu$}. One can check that
$$\rho(\mu')=a\rho(\mu)~~\ \&~~\ A_L(\mu')=a^{2}A_L(\mu).$$  This is called {\em the shifting property} of invariant probability measures for geodesic flows. The lemma in the following is a consequence of the shifting property (cf.~\cite{W}).

\begin{lem}\label{shifting property}
If $\mu$ is a minimal measure for the geodesic flow, and $\mu'$
is a shifting of $\mu$, then $\mu'$ is also a minimal measure.
\end{lem}
To establish our theory of the minimal measures distributed on closed trajectories for the geodesic flows, we define the terminology of the disjoint partition in \cite{W}.
\begin{definition}[Disjoint Partition]\label{disjiont partition}
Given a first homology class $h\in H_{1}(M,\Z)$, we say a finite set of piecewise smooth simple closed
curves $\{l_{1}, \cdots, l_{n}\}$, $n\geq 1$, where $l_i : [0,T_i]\rightarrow M$ for every $i=1,
\cdots,n$, is a disjoint partition of $h$, if
$[l_{1}]+ \cdots+ [l_{n}]=h$, and for each pair $1\leq i<j\leq n$, either $l_i=l_j$ possibly after an orientation preserving
re-parameterization or $l_{i}\cap
l_{j} =\emptyset$.
\end{definition}

If in a disjoint partition $\{l_{1}, \cdots,
l_{n}\}$ we find $l_i=l_j$ for some $1\leq i< j\leq n$ (up to an orientation preserving
re-parameterization), we can
look $l_i$ and $l_j$ as two copies of a unique closed curve. Then
sometimes we also write the disjoint partitions in the form $\{k_1 l_{1}, \cdots,
k_m l_{m}\}$ ,where $k_i \in \Z$ is the multiplicity $l_i$ in the
disjoint partition.\\

Given a disjoint partition $\mathcal{A}= \{l_{1}, \cdots, l_{n}\}$ (of some first homology class $h$), where $l_i:[0,T_i]
\rightarrow M,\ i=1,\cdots,n$, we define its total arclength $|\mathcal{A}|$ to be the summation of the arclength of its elements,  i.e
$|\mathcal{A}|=|l_1|+\cdots+|l_n|.$ Moreover we use $[\mathcal{A}]$
to denote
$[l_1]+\cdots+[l_n]=h\in H_1(M,\Z)$. We say the disjoint partition
$\mathcal{A}$ {\em supports} a probability measure $\mu$ (defined on $TM$), if
$$\pi(\mbox{supp}(\mu))=l_1\cup\cdots\cup l_n,$$ where
$\pi: (x,v)\mapsto x$ is the standard projection from $TM$ to $M$.\\

For each $h\in H_1(M,\mathbb{Z})$ there are infinitely many
disjoint partitions of $h$. Therefore the total arclength of these disjoint
partitions of $h$ varies. However they have an obvious lower bound which is
$0$. The following lemma tells that for each $h\in
H_1(M,\mathbb{Z})$, $$\min_{[\mathcal{A}]=h}
\{|\mathcal{A}|\}=\inf_{[\mathcal{A}]=h} \{|\mathcal{A}|\}.$$

\begin{lem}[cf.~\cite{W}]\label{existence of minimal disjoint partition}
For every $h\in H_{1}(M,\Z)$, there is a disjoint partition of $h$, which has minimal total
arclength among all disjoint partitions of $h$.
\end{lem}

From the lemma, we can see that the minimal total arclength among
disjoint partitions of $h$ can be realized by a disjoint partition.
We call disjoint partitions with this property {\em the minimal
disjoint partitions}. Obviously, minimal disjoint partitions are consisting of simple closed geodesics. Now let us consider a minimal measure $\mu$ of the geodesic flow which is distributed on a finite set of closed geodesics $\{(\gamma_1,\gamma'),\cdots,(\gamma_n,\gamma'_n)\}$, where $\gamma_i:\R\ra M$ is a periodic geodesic with least period $T_i>0$, $i=1,\cdots, n$. Let $l_i=\gamma_i|_{[0,T_i]}$, $i=1,\cdots,n$, then $\mathcal{A}=\{l_1,\cdots,l_n\}$ is a disjoint partition of $h=[l_i]+\cdots+[l_n]\in H_1(M,\Z)$. The following lemma tells implies that
$\mathcal{A}$ should be a minimal disjoint partition. And moreover a minimal disjoint partition always supports a minimal measure.

\begin{lem}[cf.~\cite{W}]\label{measure-partitions}
\begin{enumerate}
\item For each $h\in H_1(M,\Z)$, each disjoint partition
$\mathcal{A}$ of $h$ consisting of geodesics, and each $T>0$,
there is an invariant probability measure $\mu$ supported on
$\mathcal{A}$ with rotation vector $\frac{h}{T} \in H_1(M,\R)$,
whose action is minimal among the actions of all the invariant
probability measures supported on $\mathcal{A}$ with rotation vectors
$\frac{h}{T}$.
\item If $\mathcal{A}$ and $\mathcal{B}$ are both
disjoint partitions of $h\in H_1(M, \Z)$ consisting of geodesics,
and $\mu_A$ and $\mu_B$ are invariant probability measures with rotation vectors $\frac{h}{T}$ supported
$\mathcal{A}$ and $\mathcal{B}$ respectively, which satisfy the
minimal action property described in (1). Then
\begin{equation}\label{comparison}
A_L(\mu_A)\leq A_L(\mu_B)\Longleftrightarrow |\mathcal{A}| \leq
|\mathcal{B}|.
\end{equation}
\item For each $h\in H_{1}(M,\Z)$, and $T>0$, if $\mathcal{A}$ is a minimal disjoint partition of $h$, then the invariant measure $\mu_A$ in the above is a minimal measure with rotation vector $\frac{h}{T}\in H_1(M,\R)$.
\end{enumerate}
\end{lem}

The proof is very long and was presented in \cite{W}. We remark that the third statement in Lemma \ref{measure-partitions} is a key tool in the next section. It gives a (partial) criteria to distinguish minimal measures distributed on closed trajectories. For example, given a closed geodesic $\gamma$ with least period $T$, and let $\mu$ be the ergodic measure evenly distributed on $(\gamma,\gamma')|_{[0,T]}$.  We expect to know whether $\mu$ is a minimal measure or not. Let $l=\gamma|_{[0,T]}$, then $\mathcal{A}=\{l\}$ is a disjoint partition of $h=[l]\in H_1(M,\Z)$ (consisting of only one closed curve). Obviously $$\rho(\mu)=\rho(l)=\frac{h}{T}\in H_1(M,\R),$$ here we identify $H_{1}(M, \Z)$ with the lattice of integral vectors in $H_{1}(M, \R)$. It can be shown by the shifting property and Jesen's inequality that $\mu$ has minimal action among all the invariant probability measures supported on $\mathcal{A}$ with rotation vectors $\frac{h}{T}$. Then Lemma \ref{measure-partitions} tells us that $\mu$ is a minimal measure if $\mathcal{A}$ is a minimal disjoint partition of $h$.
\vspace{.3cm}

\section{PROOF OF THE MAIN THEOREM}\label{chapter6}
In this section, we will give the proof of Theorem \ref{main theorem}. Suppose $M$ is a compact closed orientable surface with genus $g>1$, and $G$ is a complete Riemannian metric on M.  Consider the Lagrangian system generated by $L(x,v)=G_x(v,v)$, which is exactly the geodesic flow. First of all, we should notice that if $\mu$ is evenly distributed on a closed trajectory $(\gamma, \gamma')$, and is minimal in the homotopic version, then $\gamma$ is an action-minimizer in the homotopical version, and vice versa. Suppose the least positive period of $\gamma$ is $T>0$, then $\gamma|_{[0,T]}$ is a shortest representative in its free homotopy class. It is a standard result in geometry that on compact closed Riemannian manifolds, each free homotopy class of closed curves has a shortest representative, which is a smooth closed geodesic. So this kind of minimal ergodic measures in the homotopical version are abundant.\\

Let $l=\gamma|_{[0,T]}: [0,T]\ra M$. In the following we can assume $l$ is a simple closed curve, for otherwise, if $l$ has self-intersections, we can find a finite-fold covering space on which any lifting of $l$ is simple closed. Let $h=[l]\in H_1(M,\Z)$. If $h\neq e$  and $l$ has the shortest arclength among all disjoint partitions of $h$, then by Lemma \ref{measure-partitions} (3), $\mu$ is already a minimal measure in Mather theory. So we just have to take the covering map to be the trivial covering $p=\mbox{id}:M\ra M$.\\

We remark that Lemma \ref{measure-partitions} (3) is the key observation, $\mu$ is a minimal measure or not completely determined by whether or not $l$ is a minimal disjoint partition in its homology class. So if $l$ is not a minimal disjoint partition, we will show that \emph{there is a finite-fold covering space $M'$ with the covering map $p: M'\ra M$, such that each preimage of $l$ on $M'$ is a minimal disjoint partition in its homology class in $H_1(M',\Z)$}. We will show this based on the following three lemmas.

\begin{lem}\label{nontrivial geodesic}
If $[l]=e\in H_1(M,\Z)$, and $l$ is homotopically nontrivial, then there is a finite-fold covering space $M_1$ with the covering map $p_1: M_1\ra M$ such that
\begin{enumerate}
\item Each preimage of $l$ is still a simple closed curve.
\item Each preimage of $l$ is homologically nontrivial, i.e.~if $l_1$ is a preimage of $l$, then $[l_1]\neq e \in H_1(M_1,\Z)$.
\end{enumerate}
\end{lem}
\begin{proof}
By the assumptions, $M$ is a compact closed surface with genus $g\geq 2$ and $l\subset M$ is trivial in homology but nontrivial in homotopy. Therefore $M-l$ is disconnected. And moreover we have two compact surfaces $S_1$ and $S_2$, each of which has a boundary homeomorphic to $\mathbb{S}^1$ and has positive genus, such that $M=S_1\cup S_2$ and $l=S_1\cap S_2$. Suppose $S_1$ has genus $m$ with $1\leq m<g$, and $S_2$ has genus $1\leq g-m<g$.\\

Take a point $x\in l$, then the fundamental group based on $x$ can be expressed as $$\pi_1(M,x)=\langle a_1,a_2,\cdots,a_{2g-1},a_{2g}~|~[a_1,a_2]\cdots[a_{2g-1},a_{2g}]=e\rangle,$$
where the first $2m$ generators can be represented by loops in $S_1$ and the rest $2g-2m$ of the generators can be represented by loops in $S_2$. Therefore $$[l]=[a_1,a_2]\cdots[a_{2m-1},a_{2m}]\in \pi_1(M,x).$$ \vspace{.1ex}

Fix the set of generators $\{a_1,a_2,\cdots,a_{2g-1},a_{2g}\}$ of $\pi_1(M,x)$. For each homotopy class $a=a_{j_1}^{\varepsilon_{j_1}}\cdots a_{j_N}^{\varepsilon_{j_N}}\in \pi_1(M,x)$, $\varepsilon_{j_1}, \cdots, \varepsilon_{j_N} \in \Z$, we define $$\chi_t(a)=\sum_{j_i=t} \varepsilon_{j_i},~~t=1,\cdots, 2g.$$ One can check that the values of $\chi_t(a)$'s do not depend on the choice of expressions of $a$,  i.e.~although the expression of $a$ under this set of generators may not be unique, $\chi_t(a)$ is uniquely determined for every $t=1,\cdots, 2g$.  Let $$G=\{a\in \pi_1(M,x)~|~\chi_{2m-1}(a)+\chi_{2m+1}(a)=0\mod 2\}.$$ Then we can see that:
\begin{enumerate}
\item $G$ is a normal subgroup of $\pi_1(M,x)$. This is because for any $a\in G$ and $b\in  \pi_1(M,x)$, $\chi_t(bab^{-1})=\chi_t(a)$, $t=1,\cdots,2g$, therefore $$a\in G \Longrightarrow bab^{-1}\in G,~\forall b\in \pi_1(M,x).$$
\item $[\pi_1(M,x):G]=2$.
\item $[l]\in G$. Since $[l]=[a_1,a_2]\cdots[a_{2m-1},a_{2m}]$,  $\chi_{2m-1}([l])=\chi_{2m+1}([l])=0$.
\item $a_{2m-1}, a_{2m+1}\notin G$.
\end{enumerate}\vspace{1ex}

By the standard theory of the existence of covering spaces (cf.~\cite{Ro}), there is a 2-fold covering space $M_1$ with the covering map $p_1: M_1\ra M$ such that $$p_{1*}(\pi_1(M_1,x_1))=G,$$ where $x_1\in M_1$ is an arbitrary preimage of $x$ under $p_1$.\\

Let $l_1\subset M_1$ be the preimage of $l$ under $p_1$ passing through $x_1$. We know $l_1$ is still a closed curve since $[l]\in G=p_{1*}(\pi_1(M_1,x_1))$. By consider the Deck transformations, we know that the other preimage $l_2$ of $l$ is also a closed curve. To show that $l_1$ is nontrivial in homology, we just have to show that $M_1-l_1$ is connected.\\

Let $l_{2m-1},~l_{2m+1}: [0,1]\ra M$ be representatives of $a_{2m-1}, a_{2m+1}$ respectively, such that $l_{2m-1}((0,1))\subset \mbox{int}(S_1)$, $l_{2m+1}((0,1))\subset \mbox{int}(S_2)$. Consider $L=l_{2m-1}*l_{2m+1}:[0,1]\ra M$, then $L\cap l=\{x\}=\{L(0)\}=\{L(1)\}=\{L(\frac{1}{2})\}$. Moreover we can require $L$ intersects $l$ transversally by considering well-chosen $l_{2m-1},~l_{2m+1}$. Let $L_1:[0,1]\ra M_1$ be the lifting of $L$ with $L_1(0)=L_1(1)=x_1$. Since $a_{2m-1}, a_{2m+1}\notin G$, we know that $L_1(\frac{1}{2})\neq x_1$. Therefore $L_1(\frac{1}{2})$ should be another preimage of $x$, called $x_2$. Obviously $x_2\in l_2$. So $L_1$ connects the two brunches $l_1$ and $l_2$ of the preimage of $l$, and intersects them transversally at exact $x_1,x_2$ respectively. Then for $\varepsilon>0$ small enough, $L_1|_{[\varepsilon, 1-\varepsilon]}$ is a continuous path connecting two points, which are located on different side of $l_1$. Therefore $M_1-l_1$ is path connected. So $l_1$ is nontrivial in homology.
\end{proof}

We remark that we can explain the construction of the covering space $M_1$ in a more straightforward way. Since $l$ is trivial in homology but nontrivial in homotopy, we can find a homologically nontrivial simple closed curve $\gamma$ on $M$, such that $\gamma$ transversally intersects $l$ at exactly two distinct points. For example, $\gamma$ can be chosen to be homotopic to a representative of the homotopy class $a_{2m}a_{2m+2}\in\pi_1(M,x)$.  \\

Now take two copies of $M$ and cut both of them along $\gamma$. Since $\gamma$ is nontrivial in homology, after this cutting, we get two copies of a connected surface with two boundaries $c_1$ and $c_2$, each of which is homeomorphic to $\mathbb{S}^1$. We attach $c_1$ of the first copy to $c_2$ of the second copy, and attach $c_2$ of the second copy to $c_1$ of the first copy. Then, we get a compact connected manifold $M_1$ without boundary, which is a two-fold covering space of $M$. Denote the covering map as $p_1: M_1\ra M$.
\vspace{-1cm}
\begin{center}
\hspace{-6ex}\includegraphics[width=12cm]{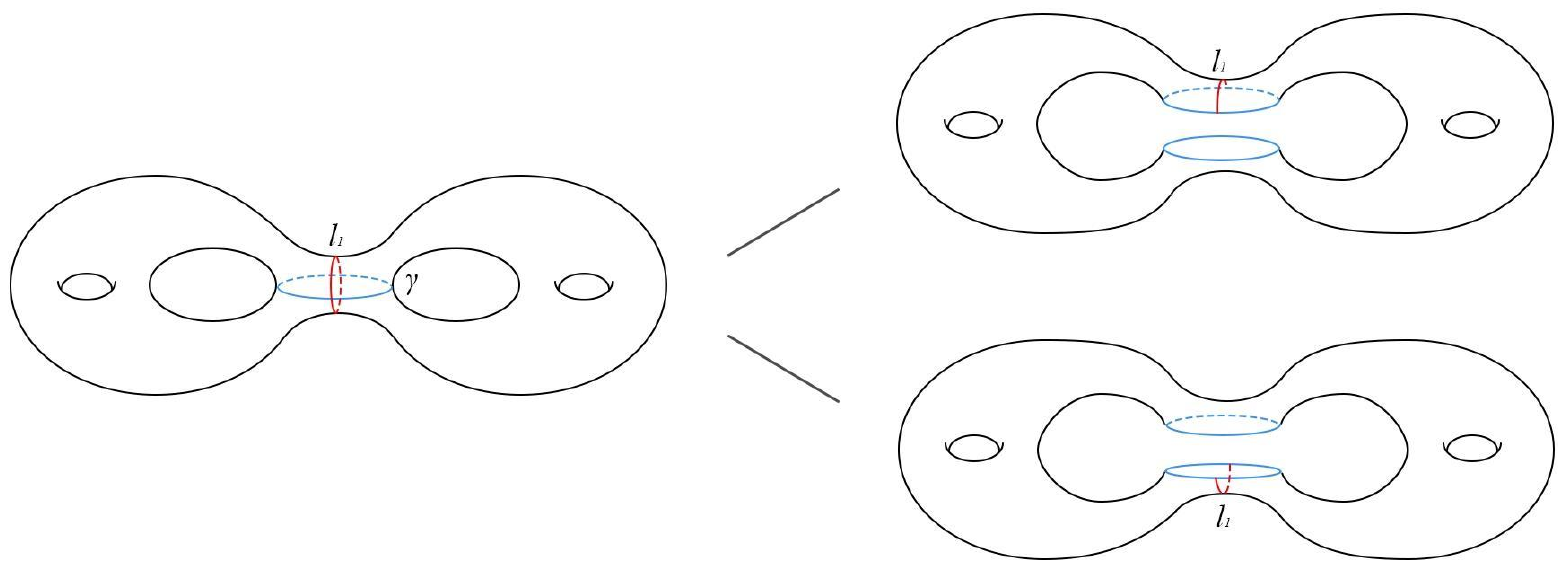}
\end{center}
\vspace{-1cm}
$l$ has two $p_1$-preimage on $M_1$, both are closed curves. Consider an arbitrary lifting of $l$, called $l_1$. If we cut $M_1$ along $l_1$, we get a connected surface with two boundaries homeomorphic to $\mathbb{S}^1$ (if we want to split $M_1$ into two parts, we have to cut $M$ along the union of the two preimages of $l$). This implies that $l$ is nontrivial in homology.\\

\begin{lem}\label{shortest geodesic}
Suppose $[l]\neq e\in H_1(M,\Z)$ and $l'$ is a simple closed geodesic, which is homologous to $l$ but not homotopic to $l$. Then there is a finite-fold covering space $M_2$ with covering map $p_2: M_2\ra M$ such that
\begin{enumerate}
\item Each preimage of $l$ is still a simple closed curve.
\item None of preimages of $l$ and preimages of $l'$ are homologous on $M_2$, i.e.~if $l_1$ is a preimage of $l$, $l'_1$ is a preimage of $l'$, then $[l_1]\neq [l'_1] \in H_1(M_2,\Z)$.
\end{enumerate}
\end{lem}
\begin{proof}
This is actually a corollary of the previous lemma. Note that by the homotopy lifting theorem (cf.~\cite{Ro}), If $l$ and $l'$ are homotopic to each other, we can not break this relation by lifting them to any covering spaces (no matter finite-fold or infinite-fold covering). However, if $l$ and $l'$ are just homologous but not homotopic to each other, the situation would be completely different. \\

Since $l'$ is homologous to $l$ but not homotopic to $l$, their difference in homotopy is a nontrivial commutator element. To be more precise, take a point $x\in l$ and let $s$ be a closed curve passing through $x$ and homotopic to $l'$, then the equivalence class $\alpha=[l*s^{-1}]\in \pi_1(M,x)$ is a nontrivial commutator element.\\

By Lemma \ref{nontrivial geodesic} we can find a 2-fold covering $M_2$, on which the preimages of $l*s^{-1}$ is nontrivial in homology. Therefore each preimages of $l$ and preimages of $l'$ are no longer homologous on $M$ (if both of them are still closed curves). To guarantee that preimages of $l$ are still simple closed curves, we just need to choose a special cutting $\gamma$, which does not intersect $l$. This is always feasible.

\end{proof}

We say a disjoint partition $\mathcal{A}$ of $h\neq e\in H_1(M,\Z)$ is nontrivial if it is consisting of more than one distinct closed curves.

\begin{lem}\label{shortest partition}
Suppose $h=[l]\neq e\in H_1(M,\Z)$ and $\mathcal{A}$ is a nontrivial disjoint partition of  $h$. Then there is a finite-fold cover space $M_3$ with covering map $p_3: M_3\ra M$ such that
\begin{enumerate}
\item Each preimage of $l$ is still a simple closed curve.
\item For each preimage $c$ of  $l$ on $M_3$, there is no disjoint partition of $[c]\in H_1(M_3,\Z)$ whose projection on $M$ is $\mathcal{A}$.
\end{enumerate}

\end{lem}
\begin{proof}
The proof of this lemma is more difficult than the previous ones. In fact, a similar result has been presented in our previous work \cite{WX}, at here we will use the same idea in this proof. Suppose $\mathcal{A}=\{n_1l_1,\cdots,n_{k}l_{k}\}$, $k\geq 2$. Since $\mathcal{A}$ is a disjoint partition of $h$, then we have $$n_1[l_1]+\cdots+n_k[l_k]=h\in H_1(M,\Z).$$  Assume each pair of $l_i$, $l_j$ are not
homologous for $1\leq i< j\leq k$, the case that some pairs are homologous is more complicated but it can be solved using the same idea. Since $l_1,\cdots, l_k$ are pairwise disjoint and non-homologous nontrivial simple closed curve, we can choose a special set of generators of $H_1(M,\Z)$, such that $[l_1],\cdots,[l_k]\in H_1(M,\Z)$ are elements in this set of generators.\\

Take an arbitrary point $x\in l\subset M$, consider the homotopy class $\alpha_0=[l]\in \pi_1(M,x)$.  Let $s_1,\cdots,s_k$ be a set of simple closed curves passing through $x$ such that $s_i$ is homotopic to $l_i$ for all $i=1,\cdots,k$. Let $\{a_1,\cdots,a_{2g}\}\subset \pi_1(M,x)$ to be a set of generators of $\pi_1(M,x)$ such that $a_i=[s_i]\in\pi_1(M,x),~\forall~i=1,\cdots,k$. Then under this set of generators, $\alpha_0$ can be expressed in the shortest form as $$\alpha_0=a_{i_1}^{\varepsilon_1}a_{i_2}^{\varepsilon_2}\cdots a_{i_m}^{\varepsilon_m},$$ where $i_j\in\{1,\cdots,2g\},~j=1,\cdots,m$ and $\varepsilon_1,\cdots,\varepsilon_m\in \Z$. \\

For every homotopy class $a=a_{j_1}^{\varepsilon_{j_1}}\cdots a_{j_N}^{\varepsilon_{j_N}}\in \pi_1(M,x)$, we define $$\chi_t(a)=\sum_{j_i=t} \varepsilon_{j_i},~~t=1,\cdots, 2g.$$ Obviously, $\chi_t(a)$'s do not depend on the choice of expression of $a$. Then by the choice of this set of generators, we know that
$$\chi_t(\alpha_0)=n_t,~\forall~t\leq k;\hspace{.5cm}\&\hspace{.5cm}\chi_t(\alpha_0)=0, ~\forall~t>k.$$

Consider the subset:
\begin{equation*}
G=\{a\in \pi_1(M,x)~|~\chi_1(a)=t_1n_1,~\chi_2(a)=t_2n_2,~t_1+t_2\equiv 0\mod 2\}.
\end{equation*}
Then we have that:
\begin{enumerate}
\item $G$ is a normal subgroup of $\pi_1(M,x)$. This is because for any $a\in G$ and $b\in  \pi_1(M,x)$, $\chi_i(bab^{-1})=\chi_i(a)$, $i=1,2$.
\item $[\pi_1(M,x):G]=\varepsilon_1\times\varepsilon_2-2$. This is showed in \cite{WX} in the proof of Lemma 5.4.
\item $\alpha_0\in G$, since $\chi_1(\alpha_0)=n_1$, $\chi_2(\alpha_0)=n_2$.
\item $a_1^{m_1}, a_2^{m_2}\notin G$, for all $0<|m_1|<2n_1$ and $0<|m_2|<2n_2$, this can be seen by checking $\chi_t(a_1^{m_1})$ and $\chi_t(a_2^{m_2})$ for $t=1,2$.
\end{enumerate}\vspace{1ex}

So, there is a covering space $M_3$ with the covering map  $p_3: M_3\ra M$ such that
\begin{itemize}
\item $p_{3*}(\pi_1(M_3,z))=G$, where $z$ is an arbitrary preimage of $x$ under $p_3$.
\item $p_3: M_3\ra M$ is a finite-fold covering.
\end{itemize}

Since $\alpha_0\in G$, there is an element $\alpha'_0\in \pi_1(M_3,z)$ with $p_{3*}(\alpha'_0)=\alpha_0\in\pi_1(M,x)$. This implies that $l$ can be lifted to a simple closed curve $c\subset M_3$ passing through $z$, which is a representative of the class $\alpha'_0\in \pi_1(M_3,z)$. By considering the Deck transformations, it is easy to see all the preimages of $l$ are simple closed curves. \\

Next we show that for each preimage $c$ of $l$, there is no disjoint partition of $h'=[c]\in H_1(M_3,\Z)$ whose projection on $M$ is $\mathcal{A}$. We show this by contradiction. Assume there is a disjoint partition $\mathcal{B}=\{m_1\gamma_1,\cdots,m_J\gamma_J\}$ of $h'$ whose projection under $p_3$ is $\mathcal{A}$. We know that for each $\gamma_j$, $p_3(\gamma_j)=k_{j,i}l_i$ for some $l_i\in \mathcal{A}$ and $0<|k_{j,i}|\leq n_i$. Moreover, since $p_3(\mathcal{B})=\mathcal{A}$, there is at least one $j\in[1,\cdots,J]$ satisfying that $p_3(l_j)=k_{j,1}l_1$, for some $k_{j,1}$ with $0<|k_{j,1}|\leq n_1$. However since $|k_{j,1}|\leq n_1$, $a_1^{k_{j,1}}\notin G=p_{3*}(\pi_1(M_3,z))$. Then by the construction of $M_3$, there is no such simple closed curve $l_j$ on $M_3$. This is a contradiction.  We are done with the proof of this lemma.

\end{proof}

Now we are ready to show the proof of Theorem \ref{main theorem}. Recall that we denote $h=[l]\in H_1(M,\Z)$. If $h=e\in H_1(M,\Z)$, then by Lemma \ref{nontrivial geodesic} we can lift $l$ to a 2-fold covering space $M_1$ on which each preimage of $l$ is a homologically nontrivial simple closed curve. So in the following, we can assume $h\neq e\in H_1(M,\Z)$.\\

By lifting the geodesic flow to the universal covering space, we can observe that there are at most finitely many homotopy classes whose shortest representatives are shorter than $l$. Suppose these homotopy classes are $\alpha_1,\cdots,\alpha_k$, and $\{l_1,\cdots,l_k\}$ is a set of shortest representatives of $\alpha_1,\cdots,\alpha_k$ respectively. Note that the choice of  $l_1,\cdots,l_k$ is not unique if one of these classes has more than one shortest representatives. But if $l_i$ and $l'_i$ are two shortest representatives of $\alpha_i$, then they are homotopic to each other. So if on some covering space, a preimage of $l$ is not homologous to any of the liftings of $l_i$, then it is not homologous to any of the liftings of $l'_i$ too. \\

Apply Lemma \ref{shortest geodesic} iteratively, we can find a sequence of 2-fold coverings $$M'_k\ra M'_{k-1}\ra\cdots\ra M'_1\ra M,$$ such that on $M'_j$, any lifting of $l$ is a simple closed curve and is not homologous to any of the liftings of $l_1,\cdots,l_j$, $1\leq j\leq k$. Then, on $M'_k$ there is no simple closed curve homologous to $l$ with shorter arclength. For otherwise, if $\gamma\in M'_k$ is homologous to a lifting of $l$ and has shorter arclength, without loss of generality we assume it is shortest in its homotopy class. Then, the projection of $\gamma$ is homotopic to some $l_j$, $j=1,\cdots,k$. This contradicts to the construction of the sequence of covering spaces. So in the following, we assume that $l$ has shortest arclength among all simple closed representatives of $h\in H_1(M,\Z)$.\\

The rest part is to prove that there is a finite-fold covering space on which any preimage of $l$ is a shortest disjoint partition in its homology class. We define an equivalence relation between two disjoint partitions $\mathcal{A}_1=\{n_1l_1,\cdots,n_{k_1}l_{k_1}\}$ and $\mathcal{A}_2=\{n'_1l'_1,\cdots,n'_{k_2}l'_{k_2}\}$ of $h$, where $n_1,\cdots, n_{k_1}, n'_1,\cdots, n'_{k_2}\in\Z^{+}$, in the following way: we say $\mathcal{A}_1\sim\mathcal{A}_2$ if
\begin{enumerate}
\item $k_1=k_2$,
\item Up to a permutation of $\{l'_1,\cdots,l'_{k_2}\}$, $l_i$ is homotopic to $l'_i$ and $n_i=n'_i$ for all $i=1,\cdots, k_1$.
\end{enumerate}
It is easy to check that $\sim$ is an equivalence relation between disjoint partitions. For a given disjoint partition $\mathcal{A}$, we denote its equivalence class as $\Sigma_{\mathcal{A}}$. \\

For each equivalence of disjoint partitions, we define the length of this equivalence class to be the minimal total arclength among disjoint partitions in this equivalence class. It is easy to see that this minimal total arclength can be realized by a disjoint partition in its class. Similar to the previous discussion, we know that there are only finitely many equivalence classes of nontrivial disjoint partitions whose shortest total arclength is shorter than $l$, denote them $\Sigma_1,\cdots,\Sigma_k$. So, by Lemma \ref{shortest partition}, we can find a sequence of finite-fold coverings $$M''_k\ra M''_{k-1}\ra\cdots\ra M''_1\ra M,$$ such that the liftings of $l$ are all simple closed curves, and for each lifting $c_j$ of $l$ on $M''_j$, there is no disjoint partition of $[c_j]\in H_1(M''_j,\Z)$ whose projection on $M$ is in any class of $\Sigma_1, \cdots, \Sigma_j$, $j=1,\cdots,k$. For more details, please refer to \cite{WX}, the proof of Lemma 5.4.\\

Notice that on $M''_k$, any lifting $c$ of $l$ is shortest among all disjoint partitions in its homology class. So by Lemma by Lemma \ref{measure-partitions} (3), the preimage of $\mu$ which is evenly distributed on $(c, c')$ is a minimal measure in Mather's definition for the lifted geodesic flow on $TM''_k$. Take $M'=M''_k$, then Theorem \ref{main theorem} is proved.
\vspace{.3cm}

\section{SOME FURTHER DISCUSSIONS}
Theorem \ref{main theorem} shows that for the geodesic flows on surfaces of higher genus, all minimal ergodic measures in homotopical version which are distributed on closed trajectories are elements of $\mathcal{M}^*_L$. To simplify the notation, we call the ergodic measures distributed on closed trajectories \emph{the periodic measures}. So every minimal periodic measure in the homopotical version for the geodesic flows can be lifted to a Mather measure on some finite-fold covering spaces. In this section, we will consider the non-periodic minimal ergodic measures in homotopical version for the geodesic flows which have some hyperbolicity. The first case we are considering is the geodesic flow on the closed surface with negative curvature. It is well-known that, restricted on each energy level (such as the unit tangent bundle $SM$), the geodesic flow is a uniformly hyperbolic flow (cf.~\cite{Anosov}).

\subsection{Geodesic flows on surfaces with negative curvature}
Suppose $M$ is a closed surface with negative curvature. Obviously, the genus of $M$ is greater than one. A key observation is that in this case all geodesics are action-minimizers in the homotopical version since the surface with negative curvature admits no conjugate points on its universal covering space. Therefore, by the definition, all invariant measures of the geodesic flow are minimal measures in the homotopical version, i.e. $$\mathcal{M}'_L=\mathfrak{M}_{inv}.$$

Suppose $\mu$ is a non-periodic ergodic measure. By the discussion above, $\mu$ is a minimal ergodic measure in the homotopical version. Without loss of generality we assume $\mbox{supp}(\mu)\subset SM$, since the situations on other energy levels are just shifting of the situation on $SM$. We can take a trajectory $(\gamma,\gamma'):\R\ra SM$ in the support of $\mu$ satisfying that: there is sequence of positive numbers $0<t_1<t_2<\cdots<t_n<\cdots$ with $t_n\ra\infty$ as $n\ra \infty$ such that
\begin{enumerate}
\item $\epsilon_n:=d((\gamma(0),\gamma'(0)),(\gamma(t_n),\gamma'(t_n)))\ra 0$, as $n\ra\infty$, where $d$ is the distance under the Sasaki metric on $TM$.
\item Let $\nu_n$ be the probability measure evenly distributed on $(\gamma(t),\gamma'(t))|_{[0,t_n]}$, then $\nu_n\ra \mu$ as $n\ra\infty$.
\end{enumerate}
\vspace{1ex}

Since $\epsilon\ra 0$, by Anosov closing lemma (cf.~\cite{Anosov}), for sufficiently large $n$, there is a closed trajectory $(c_n, c'_n):\R\ra SM$, whose period $T_n$ satisfies that $|T_n-t_n|\leq 2 \epsilon$, such that $$d((\gamma(t),\gamma'(t)),(c(t),c'(t)))<\delta_n$$ for all $0\leq t\leq \min\{t_n, T_n\}$, for some small constant $\delta_n$ with $\delta_n\ra 0$ as $n\ra \infty$. Let $\mu_n$ be the ergodic measure evenly distributed on $(c_n,c'_n)$. Then one can check that $$\textbf{d}(\mu_n,\nu_n)\ra 0,~\mbox{as}~ n\ra\infty,$$ where $\textbf{d}$ is the distance in the space of probability measures which is compatible the the weak-* topology (cf.~\cite{Walters}). Therefore we have: $$\mu_n\ra \mu,~\mbox{as}~n\ra\infty.$$

Notice that each $\mu_n$ is distributed on a closed trajectory, then $\mu_n$ is a minimal measure in the homotopical version for every $n$. By Theorem \ref{main theorem}, we know that $\mu_n\in\mathcal{M}^*_L$. Therefore $\mu\in \overline{\mathcal{M}^*_L}$. Combining with case of periodic ergodic measures we have considered in Theorem \ref{main theorem}, we know that for the geodesic flows on compact surfaces with negative curvature, all minimal ergodic measures in the homotopical version, no matter they are periodic or non-periodic, are in the closure of $\mathcal{M}^*_L$. \\

Now we consider general invariant measures. It is showed in \cite{Sigmund} that, on the unit tangent bundle $SM$, the set of ergodic measures is a dense $G_{\delta}$ subset of the set of invariant measure. This proof is based on the specification property which is a consequence of the uniform hyperbolicity of the geodesic flows on negatively curved surfaces. So all invariant measures can be approximated by ergodic measures, and then approximated by periodic measures. By the shifting property, this result is valid on all energy levels. To conclude, we get that $\overline{\mathcal{M}^*_L}=\mathfrak{M}_{inv}=\mathcal{M}'_L$. This is the following theorem:

\begin{thm}\label{7.1}
For the geodesic flows on compact surfaces with negative curvature, $\mathcal{M}^{*}_{L}$ is a dense subset of $\mathfrak{M}_{inv}$.
\end{thm}
\vspace{.5ex}

\subsection{Geodesic flows on rank 1 surfaces with non-positive curvature}
Now we consider a geodesic flow with weaker hyperbolicity. Suppose $M$ is a closed surface of genus $g>1$, equipped with a Riemannian metric $G$ whose curvature is everywhere non-positive. We know that, in this case, all geodesics on $M$ are action-minimizers in the homotopical version since the surface with non-positive curvature does not admit any conjugate points on its universal covering space. Therefore we still have $\mathcal{M}'_L=\mathfrak{M}_{inv}$. Since we are going to consider properties related to the entropy, so in this subsection we restrict the geodesic flow on the unit tangent bundle $SM$ which is compact. We use $\mathcal{M}_L(1), \mathcal{M}^*_L(1), \mathcal{M}'_L(1)$ to denote the subset of $\mathcal{M}_L, \mathcal{M}^*_L, \mathcal{M}'_L$ whose elements are supported on $SM$, respectively.\\

In addition, we assume $(M, G)$ is a rank 1 surface, which means that it admit at least one rank 1 geodesic. Here we say a geodesic $\gamma$ is of \emph{rank 1} if there is no parallel perpendicular Jacobi field along $\gamma$, and a unit tangent vector $v$ is of \emph{rank 1} if the geodesic determined by $v$ is of rank 1. We use $\textbf{Reg}$ to denote the set of all rank 1 unit tangent vectors and $\textbf{Sing}$ to denote the compliment of $\textbf{Reg}$. Due to the existence of points with zero curvature, the geodesic flow is no longer uniformly hyperbolic. However, on rank 1 manifolds with non-positive curvature, the geodesic flows are non-uniformly hyperbolic, and many hyperbolic properties are still valid if we restrict the geodesic flow to $\textbf{Reg}$. In \cite{CS}, Coud\`{e}ne and Schapira showed that all invariant measures on $\textbf{Reg}$ can be approximated by periodic measures. Then, by Theorem \ref{main theorem}, $\overline{\mathcal{M}^*_L}$ contains all invariant measures supported on $\textbf{Reg}$, i.e.~ $$\mathfrak{M}_{inv}(\textbf{Reg})\subset\overline{\mathcal{M}^*_L(1)}.$$

Since the genus of $M$ is at least 2, by \cite{Dina}, the geodesic flow restricted on $SM$ has a positive topological entropy $h>0$. In \cite{Kn} Knieper proved that there is a unique measure of maximal entropy $\mu_{max}$ whose support is inside $\textbf{Reg}$.   Therefore, $\mu_{max}$ can be approximated by periodic measures, i.e. $\mu_{max}\in\overline{\mathcal{M}^*_L(1)}$.  So we have $$\max\{h_{meas}(\mu)~|~\mu\in\overline{\mathcal{M}^*_L(1)}\}=h.$$
Moreover, since periodic measures have 0 entropy and are dense in $\mathfrak{M}_{inv}(\textbf{Reg})$, it is straightforward that the measure-theoretic entropies of the elements in $\overline{\mathcal{M}^*_L(1)}$ will cover the whole interval $[0,h]$. \\

In addition, we know that the topological entropy of the geodesic flow restricted on $\textbf{Sing}$ is strictly less than $h$ (the entropy gap, cf.~\cite{Kn}). Then ergodic measures with entropy greater than the topological entropy on $\textbf{Sing}$ are all supported on $\textbf{Reg}$, therefore are contained in $\overline{\mathcal{M}^*_L(1)}$. In short, ergodic measures with large entropy can be approximated by projections of classic minimal measures on finite-fold covering spaces.\\

We conclude the above results in the following theorem.
\begin{thm}\label{7.2}
Suppose $(M,G)$ is a rank 1 surface of higher genus with non-positive curvature. Let $h>0$ be the topological entropy of the geodesic flow on $SM$. Then the following properties hold:
\begin{enumerate}
\item $\mathcal{M}^{*}_{L}(1)\cap \mathfrak{M}_{inv}(\textbf{Reg})$ is a dense subset of $\mathfrak{M}_{inv}(\textbf{Reg})$.
\item For each $a\in[0,h]$, there is an invariant measure $\mu\in\overline{\mathcal{M}^*_L(1)}$ satisfying that $h_{meas}(\mu)=a$.
\item There is a constant $h'\in(0,h)$ such that for any ergodic measure $\mu$ with $h_{meas}(\mu)>h'$, we have $\mu\in\overline{\mathcal{M}^*_L(1)}$.
\end{enumerate}
\end{thm}
We remark that all the results in Theorem \ref{7.2} can be extended to the geodesic flows on rank 1 surfaces without focal points, based on recent works \cite{LWW1,LLW}. We omit this discussion at here.

\vspace{.5cm}

\textbf{Acknowledgements.} We would like to thank Professor Xuezhi Zhao at Capital Normal University for some useful comments and suggestions.


\vspace{.5cm}







\begin{thebibliography}{23}

\bibitem{Anosov}
D.~V.~Anosov, Geodesic flows on closed Riemann manifolds with negative curvature, {\em Proceedings of the Steklov Institute of Mathematics}, No. 90 (1967).

\bibitem{CS}
Y.~Coud\`{e}ne and B.~Schapira, Generic measures for geodesic flows on nonpositively curved manifolds, \emph{J. c. polytech. Math.}, \textbf{1} (2014), 387C-408.

\bibitem{Dina}
E.~I.~Dinaburg, On the relations among various entropy characteristics of dynamical systems. \emph{Math USSR Izv}, \textbf{5} (1971),
337--378.


\bibitem{FJ}
T.~Farrell and L.~Jones, A topological analogue of Mostow's rigidity theorem, \emph{J. Amer. Math. Soc.}, \textbf{2} (1989), 257--370.


\bibitem{Kn}
G.~Knieper, The uniqueness of the measure of maximal entropy for geodesic flows on rank manifolds, \emph{Ann. of Math.},  \textbf{48} (1998), 291--314.

\bibitem{LLW}
F.~Liu, X.~Liu and F~ Wang, Some dynamical properties on manifolds with no conjugate points, preprint, arXiv:2105.07170.

\bibitem{LWW}
F.~Liu, F.~Wang and W.~Wu, The topological entropy for autonomous Lagrangian systems on compact manifolds whose fundamental groups have exponential growth, \emph{ Sci. China Math.}, \textbf{63} (2020), 1323--1338.

\bibitem{LWW1}
F.~Liu, F.~Wang and W.~Wu, On the Patterson-Sullivan measure for geodesic flows on rank 1 manifolds without focal points, \emph{Discrete Contin. Dyn. Syst.}, \textbf{40} (2020), 1517--1554.

\bibitem{Mat2}
J.~Mather, Action minimizing invariant probability measures for positive definite Lagrangian systems, {\em Math. Z.}, {\bf 207}
(1991), 169--207.


\bibitem{Pa}
G.~Paternain, Geodesic Flows, Progress in Mathematics, 180. Birkh\"auser,  (1999).

\bibitem{Ro}
J.~Rotman, An Introduction to Algebraic Topology, Springer, (1988).

\bibitem{Sigmund}
K.~Sigmund, On the space of invariant measures for hyperbolic flows, \emph{Amer. J. Math.}, \textbf{94} (1972), 31--37.

\bibitem{Walters}
P.~Walters, An Introduction to Ergodic Theory, Springer-Verlag, (1982).

\bibitem{W}
F.~Wang, Minimal measure on surfaces of higher genus, \emph{ J. Differential Equations}, \textbf{249} (2009), 3258-3282.

\bibitem{WX}
F.~Wang and Z.~Xia, Minimal measures for Euler-Lagrange flows on finite covering spaces, \emph{Nonlinearity}, \textbf{29} (2016), 3625--3646.
\end{thebibliography}
\end{document}